\documentclass[12pt, reqno]{amsart}
\usepackage{amsmath, amsthm, amscd, amsfonts, amssymb, graphicx, color}
\usepackage[bookmarksnumbered, colorlinks, plainpages]{hyperref}
\hypersetup{colorlinks=true,linkcolor=red, anchorcolor=green, citecolor=cyan, urlcolor=red, filecolor=magenta, pdftoolbar=true}

\newtheorem{theorem}{Theorem}[section]

\newtheorem{proposition}[theorem]{Proposition}

\theoremstyle{definition}

\theoremstyle{remark}
\newtheorem{remark}[theorem]{Remark}
\numberwithin{equation}{section}

\newcommand{\be}{\begin{equation}}
\newcommand{\ee}{\end{equation}}

\newcommand{\cM}{{\mathcal M}}
\newcommand{\NN}{\mathbb{N}}

\begin{document}
\setcounter{page}{1}

\title[Inequalities and equalities on the joint and generalized spectral radius]{Inequalities and equalities on the joint and generalized spectral and essential spectral radius  of the Hadamard geometric mean of bounded sets of positive kernel operators}

\author{ Katarina Bogdanovi\'{c}$^{1}$, Aljo\v{s}a Peperko$^{2,3}$}
\date{\thanks{
Faculty of Mathematics$^1$,
University of Belgrade,
Studentski trg 16,
SRB-11000 Belgrade, Serbia,
email:   katarinabgd77@gmail.com \\
Faculty of Mechanical Engineering$^2$,
University of Ljubljana,
A\v{s}ker\v{c}eva 6,
SI-1000 Ljubljana, Slovenia,\\
Institute of Mathematics, Physics and Mechanics$^3$,
Jadranska 19,
SI-1000 Ljubljana, Slovenia \\
e-mail:   aljosa.peperko@fs.uni-lj.si
} \today}



\subjclass[2020]{47A10, 47B65, 47B34, 15A42, 15A60, 15B48}

\keywords{Hadamard-Schur geometric mean; Hadamard-Schur product; joint and generalized spectral radius; essential spectral radius; measure of noncompactness;
 positive kernel operators; non-negative matrices; bounded sets of operators}


\begin{abstract}
We prove
 new inequalities and equalities for the generalized and the joint spectral radius (and their essential versions) of Hadamard (Schur) geometric means of bounded sets of positive kernel operators on  Banach function spaces. In the case  of nonnegative matrices that define operators on  Banach sequences we obtain additional results. Our results extend  results of several authors that appeared relatively recently.
\end{abstract} \maketitle

\section{Introduction}

In \cite{Zh09}, X. Zhan conjectured that, for non-negative $N\times N$ matrices $A$ and $B$, the spectral radius $\rho (A\circ B)$ of the Hadamard product satisfies
\be
\rho (A\circ B) \le \rho (AB),
\label{qu}
\ee
where $AB$ denotes the usual matrix product of $A$ and $B$. This conjecture was confirmed by K.M.R. Audenaert in \cite{Au10}
by proving
\be
\rho (A\circ B) \le \rho ((A\circ A)(B\circ B)) ^{\frac{1}{2}}\le \rho (AB).
\label{Aud}
\ee
These inequalities were established via a trace description of the spectral radius.
Soon after, inequality (\ref{qu}) was reproved, generalized and refined in different ways by
several authors (\cite{HZ10},  \cite{Hu11}, \cite{S11}, \cite{Sc11}, \cite{P12}, \cite{CZ15}, \cite{DP16}, \cite{P17}, \cite{P17+}, \cite{BP21}).
Using the fact that the Hadamard product is a principal submatrix of
the Kronecker product, R.A. Horn and F. Zhang  proved in \cite{HZ10} the inequalities
\be
\rho (A\circ B) \le \rho (AB\circ BA)^{\frac{1}{2}}\le \rho (AB).
\label{HZ}
\ee
Applying the techniques of \cite{HZ10}, Z.  Huang proved that
\be
\rho (A_1 \circ A_2 \circ \cdots \circ A_m) \le \rho (A_1 A_2 \cdots A_m)
\label{Hu}
\ee
for $n\times n$ non-negative matrices $A_1, A_2, \cdots, A_m$ (see \cite{Hu11}).  A.R. Schep was the first one to observe that the results from \cite{DP05} and \cite{P06} are applicable in this context (see \cite{S11} and \cite{Sc11}). He extended inequalities (\ref{Aud}) and (\ref{HZ}) to non-negative matrices that define bounded
operators on sequence spaces (in particular on $l^p$ spaces, $1\le p <\infty$) and proved in
\cite[Theorem 2.7]{S11}  that
\be
\rho (A\circ B) \le \rho ((A\circ A)(B\circ B))^{\frac{1}{2}}\le  \rho (AB\circ AB)^{\frac{1}{2}}\le \rho (AB)
\label{Sproved}
\ee
(note that there was an error in the statement of \cite[Theorem 2.7]{S11}, which was corrected in \cite{Sc11} and \cite{P12}).
In  \cite{P12}, the second author of the current paper extended the inequality (\ref{Hu}) to non-negative matrices that define bounded
operators on Banach sequence spaces (see below for the exact definitions) and proved that the inequalities 
\be
\rho (A\circ B) \le \rho ((A\circ A)(B\circ B))^{\frac{1}{2}}\le \rho(AB \circ AB)^{\frac{\beta}{2}} \rho(BA \circ BA)^{\frac{1-\beta}{2}} \le \rho (AB)
\label{P1}
\ee
and
\be
\rho (A\circ B) \le \rho (AB\circ BA)^{\frac{1}{2}}\le  \rho(AB \circ AB)^{\frac{1}{4}} \rho(BA \circ BA)^{\frac{1}{4}} \le  \rho (AB).
\label{P2}
\ee
hold, where $\beta \in [0,1]$. Moreover, 
 he generalized these inequalities to the setting of the generalized and the joint spectral radius of bounded sets of such non-negative matrices. 


In \cite[Theorem 2.8]{S11},  A.R. Schep proved
 that the inequality
\be
\rho \left(A ^{\left( \frac{1}{2} \right)} \circ B  ^{\left( \frac{1}{2} \right)} \right) \le \rho (AB) ^{\frac{1}{2}}
\label{Schep}
\ee
holds for positive kernel operators on $L^p$ spaces. Here $A ^{\left( \frac{1}{2} \right)} \circ B  ^{\left( \frac{1}{2}\right)} $ denotes the Hadamard geometric mean of operators $A$ and $B$. In \cite[Theorem 3.1]{DP16}, R. Drnov\v{s}ek and the second author, generalized this inequality and proved that the inequality
\be
\rho \left(A_1^{\left(\frac{1}{m}\right)} \circ A_2^{\left(\frac{1}{m}\right)} \circ \cdots \circ
A_m^{\left(\frac{1}{m}\right)}\right)   \le \rho (A_1 A_2 \cdots A_m)^{\frac{1}{m}}
\label{genHuBfs}
\ee
holds for positive kernel operators $A_1, \ldots, A_m$ on an arbitrary Banach function space $L$.
 In \cite{P17+}, the second author refined (\ref{genHuBfs}) and showed that the inequalities
$$\rho \left(A_1^{\left(\frac{1}{m}\right)} \circ A_2^{\left(\frac{1}{m}\right)} \circ \cdots \circ
A_m^{\left(\frac{1}{m}\right)}\right) \;\;\;\;\;\;\;\;\;\;\;\;\;\;\;\;\;\;\;\;\;\;\;\;\;\;\;\;\;\;$$
\be
\le \rho  \left(P_1^{\left(\frac{1}{m}\right)} \circ P_2^{\left(\frac{1}{m}\right)} \circ \cdots \circ
P_m^{\left(\frac{1}{m}\right)}\right) ^{\frac{1}{m}}  \le \rho (A_1 A_2 \cdots A_m) ^{\frac{1}{m}} .
\label{genHuBfsP}
\ee
hold, where  $P_j = A_j \ldots A_m A_1 \ldots A_{j-1}$ for $j=1,\ldots , m$. In \cite[Theorem 3.2]{P19} , the second author showed that (\ref{genHuBfsP}) holds also for the essential radius $\rho _{ess}$ under the additional condition that $L$ and its Banach dual $L^*$ have order continuous norms. Formally, here and throughout the article $A_{j-1}=I$ for $j=1$ (eventhough $I$ might not be a well defined kernel operator). In particular, the following kernel version of  (\ref{HZ}) holds:
\be
\rho \left(A^{(\frac{1}{2})} \circ B^{(\frac{1}{2})} \right) \le  \rho \left((AB)^{(\frac{1}{2})} \circ (BA)^{(\frac{1}{2})}\right)^{\frac{1}{2}} \le \rho(AB) ^{\frac{1}{2}}.
\label{sch_refin}
\ee
Several additional closely related results, generalizations and refinements of the above results were obtained in \cite{P19, Zh18, P21, BP21}.

In \cite[Theorem 3.4]{P17} and \cite[Theorem 3.5]{P19}, the second author generalized inequalities (\ref{genHuBfs}) and (\ref{sch_refin}) and their essential version to the setting of the generalized and the joint spectral radius (and their essential versions) of bounded sets of positive kernel operators on a Banach function space (see also Theorems \ref{powers} and  
\ref{refin} below). 

The rest of the article is organized in the following way. In Section 2 we recall definitions and results that we will use in our proofs. 
In Section 3 we extend the main results of \cite{P19} by proving new inequalities and equalities for the generalized and the joint spectral radius (and their essential versions) of Hadamard (Schur) geometric means of bounded sets of positive kernel operators on  Banach function spaces (Theorems \ref{kathyprop}(i), \ref{finally}(i), \ref{kathyth1}, \ref{equalities_joint} and \ref{finally2}(i)). In the case of nonnegative matrices that define operators on  Banach sequences we prove further new inequalities that extend the main results of \cite{P12} (Theorems \ref{finally}(ii), \ref{kathyth2} and \ref{finally2}(ii)). All the inequalities mentioned above are very special instances of our results. In Section 4 we prove new results on geometric symmetrization of bounded sets of positive kernel operators  on $L^2 (X, \mu)$ and on weighted geometric symmetrization of bounded sets of nonnegative matrices that define operators on $l^2(R)$, which extend some results from \cite{Drn, P21, BP21}.

\section{Preliminaries}
\vspace{1mm}

Let $\mu$ be a $\sigma$-finite positive measure on a $\sigma$-algebra $\cM$ of subsets of a non-void set $X$.
Let $M(X,\mu)$ be the vector space of all equivalence classes of (almost everywhere equal)
complex measurable functions on $X$. A Banach space $L \subseteq M(X,\mu)$ is
called a {\it Banach function space} if $f \in L$, $g \in M(X,\mu)$,
and $|g| \le |f|$ imply that $g \in L$ and $\|g\| \le \|f\|$. Throughout the article, it is assumed that  $X$ is the carrier of $L$, that is, there is no subset $Y$ of $X$ of
 strictly positive measure with the property that $f = 0$ a.e. on $Y$ for all $f \in L$ (see \cite{Za83}).

 Let $R$ denote the set $\{1, \ldots, N\}$ for some $N \in \NN$ or the set $\NN$ of all natural numbers.
Let $S(R)$ be the vector lattice of all complex sequences $(x_n)_{n\in R}$.
A Banach space $L \subseteq S(R)$ is called a {\it Banach sequence space} if $x \in S(R)$, $y \in L$
and $|x| \le |y|$ imply that $x \in L$ and $\|x\|_L \le \|y\|_L$. Observe that a Banach sequence space is a Banach function space over a measure space $(R, \mu)$,
where $\mu$ denotes the counting measure on $R$. Denote by $\mathcal{L}$ the collection of all Banach sequence spaces
$L$ satisfying the property that $e_n = \chi_{\{n\}} \in L$ and
$\|e_n\|_L=1$ for all $n \in R$. For $L\in \mathcal{L}$ the set $R$ is the carrier of $L$.

Standard examples of Banach sequence spaces are Euclidean spaces, $l^p$ spaces for $1\le p \le \infty$,  the space $c_0\in \mathcal{L}$
of all null convergent sequences  (equipped with the usual norms and the counting measure), while standard examples of Banach function spaces are  the well-known spaces $L^p (X,\mu)$ ($1\le p \le \infty$) and other less known examples such as Orlicz, Lorentz,  Marcinkiewicz  and more general  rearrangement-invariant spaces (see e.g. \cite{BS88}, \cite{CR07}, \cite{KM99} and the references cited there), which are important e.g. in interpolation theory and in the theory of partial differential equations.
 Recall that the cartesian product $L=E\times F$
of Banach function spaces is again a Banach function space, equipped with the norm
$\|(f, g)\|_L=\max \{\|f\|_E, \|g\|_F\}$.

If $\{f_n\}_{n\in \mathbb{N}} \subset M(X,\mu)$ is a decreasing sequence and
$f=\inf\{f_n \in M(X,\mu): n \in \NN \}$, then we write $f_n \downarrow f$. 
A Banach function space $L$ has an {\it order continuous norm}, if $0\le f_n \downarrow 0$
implies $\|f_n\|_L \to 0$ as $n \to \infty$. It is well known that spaces $L^p  (X,\mu)$, $1\le p< \infty$, have order continuous
norm. Moreover, the norm of any reflexive Banach function space is
order continuous. 
In particular, we will be interested in  Banach function spaces $L$ such that $L$ and its Banach dual space $L^*$ have order continuous norms. Examples of such spaces are $L^p  (X,\mu)$, $1< p< \infty$, while the space
$L=c_0$ 
is an example of a non-reflexive Banach sequence space, such that $L$ and  $L^*=l^1$ have order continuous
norms.

By an {\it operator} on a Banach function space $L$ we always mean a linear
operator on $L$.  An operator $A$ on $L$ is said to be {\it positive}
if it maps nonnegative functions to nonnegative ones, i.e., $AL_+ \subset L_+$, where $L_+$ denotes the positive cone $L_+ =\{f\in L : f\ge 0 \; \mathrm{a.e.}\}$.
Given operators $A$ and $B$ on $L$, we write $A \ge B$ if the operator $A - B$ is positive.

Recall that a positive  operator $A$ is always bounded, i.e., its operator norm
\be
\|A\|=\sup\{\|Ax\|_L : x\in L, \|x\|_L \le 1\}=\sup\{\|Ax\|_L : x\in L_+, \|x\|_L \le 1\}
\label{equiv_op}
\ee
is finite.
Also, its spectral radius $\rho (A)$ is always contained in the spectrum.

An operator $A$ on a Banach function space $L$ is called a {\it kernel operator} if
there exists a $\mu \times \mu$-measurable function
$a(x,y)$ on $X \times X$ such that, for all $f \in L$ and for almost all $x \in X$,
$$ \int_X |a(x,y) f(y)| \, d\mu(y) < \infty \ \ \ {\rm and} \ \
   (Af)(x) = \int_X a(x,y) f(y) \, d\mu(y)  .$$
One can check that a kernel operator $A$ is positive iff
its kernel $a$ is non-negative almost everywhere.

Let $L$ be a Banach function space such that $L$ and $L^*$ have order
continuous norms and let $A$ and $B$ be  positive kernel operators on $L$. By $\gamma (A)$ we denote the Hausdorff measure of
non-compactness of $A$, i.e.,
$$\gamma (A) = \inf\left\{ \delta >0 : \;\; \mathrm{there}\;\; \mathrm{is} \;\; \mathrm{a}\;\; \mathrm{finite}\;\; M \subset L \;\;\mathrm{such} \;\; \mathrm{that} \;\; A(D_L) \subset M + \delta D_L  \right\},$$
where $D_L =\{f\in L : \|f\|_L \le 1\}$. Then $\gamma (A) \le \|A\|$, $\gamma (A+B) \le \gamma (A) + \gamma (B)$, $\gamma(AB) \le \gamma (A)\gamma (B)$ and $\gamma (\alpha A) =\alpha \gamma (A)$ for $\alpha \ge 0$. Also
$0 \le A\le B$  implies $\gamma (A) \le \gamma (B)$ (see e.g. \cite[Corollary 4.3.7 and Corollary 3.7.3]{Me91}). Let $\rho _{ess} (A)$ denote the essential spectral radius of $A$, i.e., the spectral radius of the Calkin image of $A$ in the Calkin algebra. Then
\be
 \rho _{ess} (A) =\lim _{j \to \infty} \gamma (A^j)^{1/j}=\inf _{j \in \NN} \gamma (A^j)^{1/j}
\label{esslim=inf}
\ee
and $\rho _{ess} (A) \le \gamma (A)$. Recall that if $L=L^2(X, \mu)$, then $\gamma (A^*) = \gamma (A)$ and $\rho _{ess} (A^*)=\rho _{ess} (A)  $, where $A^*$ denotes the adjoint of $A$ . Note that equalities (\ref{esslim=inf}) and  $\rho _{ess} (A^*)=\rho _{ess} (A)  $ are valid for any bounded operator $A$ on a given complex Banach space $L$ (see e.g. \cite[Theorem 4.3.13 and Proposition 4.3.11]{Me91}).


Observe that (finite or infinite) non-negative matrices, that define operators on Banach sequence spaces, are a special case of positive kernel operators
(see e.g. \cite{P12}, \cite{DP16}, \cite{DP10}, \cite{P11}, \cite{BP21}, and the references cited there).

It is well-known that kernel operators play a very important, often even central, role in a variety of applications from differential and integro-differential equations, problems from physics
(in particular from thermodynamics), engineering, statistical and economic models, etc (see e.g. \cite{J82}, \cite{BP03}, \cite{LL05}, \cite{DLR13} 
and the references cited there).
For the theory of Banach function spaces and more general Banach lattices we refer the reader to the books \cite{Za83}, \cite{BS88}, \cite{AA02}, \cite{AB85}, \cite{Me91}.

Let $A$ and $B$ be positive kernel operators on a Banach function space $L$ with kernels $a$ and $b$ respectively,
and $\alpha \ge 0$.
The \textit{Hadamard (or Schur) product} $A \circ B$ of $A$ and $B$ is the kernel operator
with kernel equal to $a(x,y)b(x,y)$ at point $(x,y) \in X \times X$ which can be defined (in general)
only on some order ideal of $L$. Similarly, the \textit{Hadamard (or Schur) power}
$A^{(\alpha)}$ of $A$ is the kernel operator with kernel equal to $(a(x, y))^{\alpha}$
at point $(x,y) \in X \times X$ which can be defined only on some order ideal of $L$.

Let $A_1 ,\ldots, A_m$ be positive kernel operators on a Banach function space $L$,
and $\alpha _1, \ldots, \alpha _m$ positive numbers such that $\sum_{j=1}^m \alpha _j = 1$.
Then the {\it  Hadamard weighted geometric mean}
$A = A_1 ^{( \alpha _1)} \circ A_2 ^{(\alpha _2)} \circ \cdots \circ A_m ^{(\alpha _m)}$ of
the operators $A_1 ,\ldots, A_m$ is a positive kernel operator defined
on the whole space $L$, since $A \le \alpha _1 A_1 + \alpha _2 A_2 + \ldots + \alpha _m A_m$ by the inequality between the weighted arithmetic and geometric means.

A matrix $A=[a_{ij}]_{i,j\in R}$ is called {\it nonnegative} if $a_{ij}\ge 0$ for all $i, j \in R$.
For notational convenience, we sometimes write $a(i,j)$ instead of $a_{ij}$.

We say that a nonnegative matrix $A$ defines an operator on $L$ if $Ax \in L$ for all $x\in L$, where
$(Ax)_i = \sum _{j \in R}a_{ij}x_j$. Then $Ax \in L_+$ for all $x\in L_+$ and
so $A$ defines a positive kernel operator on $L$.

 Let us recall  the following result which was proved in \cite[Theorem 2.2]{DP05} and
\cite[Theorem 5.1 and Example 3.7]{P06} (see also e.g. \cite[Theorem 2.1]{P17}).

\begin{theorem}
\label{thbegin}
Let $\{A_{i j}\}_{i=1, j=1}^{k, m}$ be positive kernel operators on a Banach function space $L$ and let $\alpha _1$, $\alpha _2$,..., $\alpha _m$ are positive numbers.

\noindent (i) If 
$\sum_{j=1}^{m} \alpha _j = 1$, then the positive kernel operator
\be
A:= \left(A_{1 1}^{(\alpha _1)} \circ \cdots \circ A_{1 m}^{(\alpha _m)}\right) \ldots \left(A_{k 1}^{(\alpha _1)} \circ \cdots \circ A_{k m}^{(\alpha _m)} \right)
\label{osnovno}
\ee
satisfies the following inequalities
\begin{equation}
A \le
(A_{1 1} \cdots  A_{k 1})^{(\alpha _1)} \circ \cdots
\circ (A_{1 m} \cdots A_{k m})^{(\alpha _m)} , \\
\label{norm2}
\end{equation}
\begin{eqnarray}
\nonumber
\left\|A \right\| &\le &\left\|(A_{1 1} \cdots  A_{k 1})^{(\alpha _1)} \circ \cdots
\circ (A_{1 m} \cdots A_{k m})^{(\alpha _m)} \right\|\\
&\le&\left\|A_{1 1} \cdots  A_{k 1}\right\|^{\alpha _1}\cdots\left\|A_{1 m} \cdots  A_{k m}\right\|^{\alpha _m}
\label{spectral2}
\end{eqnarray}
\begin{eqnarray}
\nonumber
\rho \left(A \right)& \le &\rho \left((A_{1 1} \cdots  A_{k 1})^{(\alpha _1)} \circ \cdots
\circ (A_{1 m} \cdots A_{k m})^{(\alpha _m)}\right)\\
&\le&\rho \left( A_{1 1} \cdots  A_{k 1} \right)^{\alpha _1} \cdots
\rho \left( A_{1 m} \cdots A_{k m}\right)^{\alpha _m} .
\label{tri}
\end{eqnarray}
If, in addition, $L$ and $L^*$ have order continuous norms, then
\begin{eqnarray}
\nonumber
\gamma (A) & \le &\gamma \left((A_{1 1} \cdots  A_{k 1})^{(\alpha _1)} \circ \cdots
\circ (A_{1 m} \cdots A_{k m})^{(\alpha _m)}\right)\\
 &\le &
 \label{meas_noncomp}
\gamma (A_{1 1} \cdots  A_{k 1})^{\alpha _1} \cdots \gamma(A_{1 m} \cdots A_{k m})^{\alpha _m}, \\
\nonumber
\rho _{ess} \left(A \right) & \le &\rho _{ess} \left((A_{1 1} \cdots  A_{k 1})^{(\alpha _1)} \circ \cdots
\circ (A_{1 m} \cdots A_{k m})^{(\alpha _m)}\right)\\
&\le  &
\label{ess_spectral}
\rho _{ess} \left( A_{1 1} \cdots  A_{k 1} \right)^{\alpha _1} \cdots
\rho _{ess} \left( A_{1 m} \cdots A_{k m}\right)^{\alpha _m} .
\end{eqnarray}

\noindent(ii) If $L\in\mathcal L$, $\sum_{j=1}^{m} \alpha _j\ge 1$ and $\{A_{i j}\}_{i=1, j=1}^{k, m}$ are nonnegative matrices that define positive operators on $L$, then $A$ from (\ref{osnovno}) defines a  positive operator on $L$ and the inequalities (\ref{norm2}), (\ref{spectral2}) and (\ref{tri}) hold.
\label{DBPfs}
\end{theorem}
The following result is a special case  of Theorem \ref{DBPfs}.
\begin{theorem}
\label{special_case}
Let $A_1 ,\ldots, A_m$ be positive kernel operators on a Banach function space  $L$
and $\alpha _1, \ldots, \alpha _m$ positive numbers. 

\noindent (i) If $\sum_{j=1}^m \alpha _j = 1$, then 
\be
 \|A_1 ^{( \alpha _1)} \circ A_2 ^{(\alpha _2)} \circ \cdots \circ A_m ^{(\alpha _m)} \| \le
  \|A_1\|^{ \alpha _1}  \|A_2\|^{\alpha _2} \cdots \|A_m\|^{\alpha _m}
\label{gl1nrm}
\ee
and
\be
 \rho(A_1 ^{( \alpha _1)} \circ A_2 ^{(\alpha _2)} \circ \cdots \circ A_m ^{(\alpha _m)} ) \le
\rho(A_1)^{ \alpha _1} \, \rho(A_2)^{\alpha _2} \cdots \rho(A_m)^{\alpha _m} .
\label{gl1vecr}
\ee
If, in addition, $L$ and $L^*$ have order continuous norms, then
\be
 \gamma (A_1 ^{( \alpha _1)} \circ A_2 ^{(\alpha _2)} \circ \cdots \circ A_m ^{(\alpha _m)} )\le
  \gamma(A_1)^{ \alpha _1}  \gamma(A_2)^{\alpha _2} \cdots \gamma(A_m)^{\alpha _m}
\label{gl1meas_nonc}
\ee
and
\be
 \rho _{ess}(A_1 ^{( \alpha _1)} \circ A_2 ^{(\alpha _2)} \circ \cdots \circ A_m ^{(\alpha _m)} ) \le
\rho _{ess }(A_1)^{ \alpha _1} \, \rho _{ess}(A_2)^{\alpha _2} \cdots \rho _{ess}(A_m)^{\alpha _m} .
\label{gl1vecress}
\ee
\noindent(ii) If $L\in\mathcal L$, $\sum_{j=1}^{m} \alpha _j\ge 1$ and if  $A_1 ,\ldots, A_m$ are nonnegative matrices that define positive operators on $L$, then $A_1 ^{( \alpha _1)} \circ A_2 ^{(\alpha _2)} \circ \cdots \circ A_m ^{(\alpha _m)}$ defines a positive operator on $L$ and (\ref{gl1nrm}) and (\ref{gl1vecr}) hold.

\noindent (iii)  If $L\in\mathcal L$, $t\ge1$ and if $A, A_1 ,\ldots, A_m$ are nonnegative matrices that define operators on $L$,    then $A^{(t)}$ 
 defines an  operator on L and the following inequalities hold
\be
\;\;\;\;\;A_1^{(t)}\cdots A_m^{(t)}\le(A_1\cdots A_m)^{(t)},
\label{gl1t}
\ee
\be
\rho(A_1^{(t)}\cdots A_m^{(t)})\le\rho(A_1\cdots A_m)^{t},
\label{gl1nt}
\ee
\be
\|A_1^{(t)}\cdots A_m^{(t)}\|\le\|A_1\cdots A_m\|^{t}.
\label{gl1vecrt}
\ee

\end{theorem}

\bigskip

Let $\Sigma$ be a bounded set of bounded operators on a complex Banach space $L$.
For $m \ge 1$, let
$$\Sigma ^m =\{A_1A_2 \cdots A_m : A_i \in \Sigma\}.$$
The generalized spectral radius of $\Sigma$ is defined by
\be
\rho (\Sigma)= \limsup _{m \to \infty} \;[\sup _{A \in \Sigma ^m} \rho (A)]^{1/m}
\label{genrho}
\ee
and is equal to
$$\rho (\Sigma)= \sup _{m \in \NN} \;[\sup _{A \in \Sigma ^m} \rho (A)]^{1/m}.$$
The joint spectral radius of $\Sigma$ is defined by
\be
\hat{\rho}  (\Sigma)= \lim _{m \to \infty}[\sup _{A \in \Sigma ^m} \|A\|]^{1/m}.
\label{BW}
\ee
Similarly, the generalized essential spectral radius of $\Sigma$ is defined by
\be
\rho _{ess} (\Sigma)= \limsup _{m \to \infty} \;[\sup _{A \in \Sigma ^m} \rho _{ess} (A)]^{1/m}
\label{genrhoess}
\ee
and is equal to
$$\rho _{ess} (\Sigma)= \sup _{m \in \NN} \;[\sup _{A \in \Sigma ^m} \rho _{ess} (A)]^{1/m}.$$
The joint essential  spectral radius of $\Sigma$ is defined by
\be
\hat{\rho} _{ess}  (\Sigma)= \lim _{m \to \infty}[\sup _{A \in \Sigma ^m} \gamma (A)]^{1/m}.
\label{jointess}
\ee

It is well known that $\rho (\Sigma)= \hat{\rho}  (\Sigma)$ for a precompact nonempty set $\Sigma$ of compact operators on $L$ (see e.g. \cite{ShT00}, \cite{ShT08}, \cite{Mo}),
in particular for a bounded set of complex $n\times n$ matrices (see e.g. \cite{BW92}, \cite{E95}, \cite{SWP97}, \cite{Dai11}, \cite{MP12}).
This equality is called the Berger-Wang formula or also the
generalized spectral radius theorem (for an elegant proof in the finite dimensional case see \cite{Dai11}).
It is known that also the generalized Berger-Wang formula holds, i.e, that for any precompact nonempty  set $\Sigma$ of bounded operators on $L$ we have
$$\hat{\rho}  (\Sigma) = \max \{\rho (\Sigma), \hat{\rho} _{ess}  (\Sigma)\}$$
(see e.g.  \cite{ShT08}, \cite{Mo},  \cite{ShT00}). Observe also that it was proved in \cite{Mo} that in the definition of  $\hat{\rho} _{ess}  (\Sigma)$ one may replace the Haussdorf measure of noncompactness by several other seminorms, for instance it may be replaced by the essential norm.

In general $\rho (\Sigma)$ and $\hat{\rho}  (\Sigma)$ may differ even in the case of a bounded set $\Sigma$ of compact positive operators on $L$ (see \cite{SWP97} or also \cite{P17}).
Also, in \cite{Gui82} the reader can find an example of two positive non-compact weighted shifts $A$ and $B$ on $L=l^2$ such that $\rho(\{A,B\})=0 < \hat{\rho}(\{A,B\})$. As already noted in \cite{ShT00} also $\rho _{ess} (\Sigma)$ and $\hat{\rho} _{ess}  (\Sigma)$ may in general be different.

The theory of the generalized and the joint spectral radius has many important applications for instance to discrete and differential inclusions,
wavelets, invariant subspace theory
(see e.g. \cite{BW92}, \cite{Dai11}, \cite{Wi02}, \cite{ShT00}, \cite{ShT08} and the references cited there).
In particular, $\hat{\rho} (\Sigma)$ plays a central role in determining stability in convergence properties of discrete and differential inclusions. In this
theory the quantity $\log \hat{\rho} (\Sigma)$ is known as the maximal Lyapunov exponent (see e.g. \cite{Wi02}).

We will  use the following well known facts that hold for all $r \in \{\rho,  \hat{\rho}, \rho _{ess}, \hat{\rho} _{ess}  \}$:
\be
r (\Sigma  ^m) = r (\Sigma)^m \;\;\mathrm{and}\;\;
r (\Psi \Sigma) = r (\Sigma\Psi)
\label{again}
\ee
where $\Psi \Sigma =\{AB: A\in \Psi, B\in \Sigma\}$ and $m\in \NN$.

Let $\Psi _1, \ldots , \Psi _m$ be bounded sets of positive kernel operators on a Banach function space $L$ and let $\alpha _1, \ldots \alpha _m$ be positive numbers such that
$\sum _{i=1} ^m \alpha _i = 1$. Then the bounded set of positive kernel operators on $L$, defined by
$$\Psi _1 ^{( \alpha _1)} \circ \cdots \circ \Psi _m ^{(\alpha _m)}=\{ A_1 ^{( \alpha _1)} \circ \cdots \circ A _m ^{(\alpha _m)}: A_1\in \Psi _1, \ldots, A_m \in \Psi _m \},$$
is called the {\it weighted Hadamard (Schur) geometric mean} of sets $\Psi _1, \ldots , \Psi _m$. The set
$\Psi _1 ^{(\frac{1}{m})} \circ \cdots \circ \Psi _m ^{(\frac{1}{m})}$ is called the  {\it Hadamard (Schur) geometric mean} of sets $\Psi _1, \ldots , \Psi _m$.

The folowing result that follows from Theorem \ref{DBPfs}(i) was established in (\cite[Theorem 3.3]{P17} and \cite[Theorems 3.1 and 3.8]{P19}.  


\begin{theorem} Let $\Psi _1, \ldots , \Psi _m$ be bounded sets of positive kernel operators on a Banach function space $L$ and let
 $\alpha _1, \ldots , \alpha _m$ be positive numbers such that \\
$\sum _{i=1} ^m \alpha _i = 1$.  If  $r \in \{\rho, \hat{\rho}\}$ and $n \in \NN$, then
\be
r (\Psi _1 ^{( \alpha _1)} \circ \cdots \circ \Psi _m ^{(\alpha _m)} ) \le  r ((\Psi _1 ^n ) ^{( \alpha _1)} \circ \cdots \circ (\Psi _m ^n) ^{(\alpha _m)} ) ^{ \frac{1}{n}} \le
r(\Psi _1)^{ \alpha _1} \, \cdots r(\Psi _m)^{\alpha _m}
\label{gsh_ref}
\ee
and
\be
r \left(\Psi _1 ^{\left( \frac{1}{m} \right)} \circ \cdots \circ \Psi _m  ^{\left( \frac{1}{m} \right)} \right) \le r(\Psi _1 \Psi _2 \cdots \Psi _m) ^{\frac{1}{m}}.\;\;\;\;\;\;\;\;
\label{Hu_ess}
\ee
If, in addition, $L$ and $L^*$ have order continuous norms, then (\ref{gsh_ref}) and (\ref{Hu_ess}) hold also  for each $r\in \{ \rho _{ess}, \hat{\rho} _{ess}\}$.
\label{powers}
\end{theorem}
The following theorem \cite[Theorem 3.5]{P19} was one of the main results in \cite{P19}.
\begin{theorem} Let $\Psi $ and $\Sigma$  be bounded sets of positive kernel operators on a Banach function space $L$.
If $r \in \{\rho, \hat{\rho}\}$ and $\beta \in [0,1]$,  then we have
$$r \left(\Psi ^{\left( \frac{1}{2} \right)} \circ \Sigma  ^{\left( \frac{1}{2} \right)} \right) \le  r \left((\Psi \Sigma)^{\left(\frac{1}{2}\right)} \circ (\Sigma \Psi)^{\left(\frac{1}{2}\right)}\right) ^{\frac{1}{2}}  $$
\be
\le   r \left((\Psi \Sigma)^{\left(\frac{1}{2}\right)} \circ (\Psi \Sigma)^{\left(\frac{1}{2}\right)}\right) ^{\frac{1}{4}}   r \left((\Sigma\Psi)^{\left(\frac{1}{2}\right)} \circ ( \Sigma \Psi)^{\left(\frac{1}{2}\right)}\right) ^{\frac{1}{4}} \le r (\Psi \Sigma) ^{\frac{1}{2}},
\label{HZPep}
\ee
$$r \left(\Psi ^{\left( \frac{1}{2} \right)} \circ \Sigma  ^{\left( \frac{1}{2} \right)} \right) \le r \left( \left (\Psi ^{\left( \frac{1}{2} \right)} \circ \Psi  ^{\left( \frac{1}{2} \right)}\right) \left (\Sigma^{\left( \frac{1}{2} \right)} \circ \Sigma  ^{\left( \frac{1}{2} \right)}\right) \right)  ^{\frac{1}{2} }   $$
\be
   \le r \left((\Psi \Sigma)^{\left(\frac{1}{2}\right)} \circ (\Psi \Sigma)^{\left(\frac{1}{2}\right)}\right) ^{\frac{\beta}{2}}   r \left((\Sigma\Psi)^{\left(\frac{1}{2}\right)} \circ ( \Sigma \Psi)^{\left(\frac{1}{2}\right)}\right) ^{\frac{1-\beta}{2}} \le  r (\Psi \Sigma) ^{\frac{1}{2}}.
\label{AudSchPep}
\ee

If, in addition, $L$ and $L^*$ have order continuous norms, then (\ref{HZPep}) and (\ref{AudSchPep}) hold also for each $r\in \{ \rho _{ess}, \hat{\rho} _{ess}\}$.
\label{refin}
\end{theorem}

Given $L\in \mathcal{L}$, let  $\Psi _1, \ldots , \Psi _m$ be  bounded sets of nonnegative matrices that define operators on $L$ and let $\alpha _1, \ldots , \alpha _m$ be positive numbers such that
$\sum _{i=1} ^m \alpha _i \ge 1$. Then  the set
$$\Psi _1 ^{( \alpha _1)} \circ \cdots \circ \Psi _m ^{(\alpha _m)}=\{ A_1 ^{( \alpha _1)} \circ \cdots \circ A _m ^{(\alpha _m)}: A_1\in \Psi _1, \ldots, A_m \in \Psi _m \}$$
is a bounded set of   nonnegative matrices that define operators on $L$ by Theorem \ref{special_case}(ii).
By applying Theorem \ref{DBPfs}(ii), one can also prove the following result in a similar way as  \cite[Theorem 3.8]{P19}. We omit the details of the proof.
\begin{theorem}
\label{for_matrices}
Given $L\in\mathcal L$, let $\Psi, \Psi _1, \ldots , \Psi _m$ be bounded sets of nonnegative matrices that define operators on $L$. Let $\alpha_1, \ldots , \alpha _m$ be  positive numbers such that $\sum_{j=1}^{m}\alpha_j\ge1$, $n \in \NN$ and $r \in \{\rho, \hat{\rho}\}$. Then Inequalities (\ref{gsh_ref}) hold.

In particular, if $t\ge1$, then
\be
r (\Psi^{(t)})\le r((\Psi^n)^{(t)})^{\frac{1}{n}}\le r(\Psi)^{t}.
\label{folge}
\ee
\end{theorem}

\section{Further inequalities and equalities}

In \cite{P19} and later it remained unnoticed that several inequalities in Theorem \ref{refin}  are in fact equalities, which is established in the following result.  
\begin{theorem}
\label{kathyprop}
Let $\Psi$ and $\Sigma$ be bounded sets of positive kernel operators on a Banach function space $L$ and let  $\alpha_1, \ldots , \alpha _m$ be positive numbers such that $\sum_{j=1}^{m}\alpha_j=1$.

\noindent (i) If $r \in \{\rho, \hat{\rho}\}$ and $\beta \in [0,1]$, then
\be
r(\Psi)=r(\Psi^{(\alpha_1)}\circ\cdots\circ \Psi^{(\alpha_m)})
\label{uno}
\ee
and
\begin{eqnarray}
\nonumber
r(\Psi\Sigma)&=&r((\Psi^{(\frac{1}{2})}\circ\Psi^{(\frac{1}{2})})(\Sigma^{(\frac{1}{2})}
\circ\Sigma^{(\frac{1}{2})}))\\
&=& r \left((\Psi \Sigma)^{\left(\frac{1}{2}\right)} \circ (\Psi \Sigma)^{\left(\frac{1}{2}\right)}\right) ^{\beta}   r \left((\Sigma\Psi)^{\left(\frac{1}{2}\right)} \circ ( \Sigma \Psi)^{\left(\frac{1}{2}\right)}\right) ^{1-\beta}.
\label{duo}
\end{eqnarray}
If, in addition, $L$ and $L^*$ have order continuous norms, then (\ref{uno}) and (\ref{duo}) hold also for each $r\in \{ \rho _{ess}, \hat{\rho} _{ess}\}$.

\medskip

(ii) If $L\in\mathcal L$, $r \in \{\rho, \hat{\rho}\}$, $ m,n \in\NN$, $\alpha\ge 1$ and if $\Psi$ is 
 a  bounded set of nonnegative matrices that define operators on $L$, then
\be
r(\Psi^{(m)})\le r(\Psi\circ\cdots\circ\Psi)\le r(\Psi ^n \circ\cdots\circ\Psi^n)^{\frac{1}{n}}\le r(\Psi)^{m},
\label{tre}
\ee
where in (\ref{tre}) the Hadamard products in $\Psi\circ\cdots\circ\Psi$ and in $\Psi ^n\circ\cdots\circ\Psi ^n$ are taken $m$ times, and
\be
r(\Psi^{(\alpha)})\le r(\Psi^{(\alpha-1)}\circ\Psi)\le  r((\Psi ^n)^{(\alpha-1)}\circ\Psi ^n)^{\frac{1}{n}} \le r(\Psi)^{\alpha}.
\label{tre1}
\ee
\end{theorem}
\begin{proof}
(i) To prove (\ref{uno}) first observe that
$ \Psi\subset\Psi^{(\alpha_1)}\circ\cdots\circ\Psi^{(\alpha_m)}$, since $ A=A^{(\alpha_1)}\circ\cdots\circ
A^{(\alpha_m)}$ for all $A\in\Psi $. It follows that $$r(\Psi)\le r(\Psi^{(\alpha_1)}\circ\cdots\circ\Psi^{(\alpha_m)})
\le r(\Psi)^{\alpha_1}\cdots r(\Psi)^{\alpha_m}=r(\Psi)$$ by Theorem \ref{powers} and so $r(\Psi)=r(\Psi^{(\alpha_1)}\circ\cdots\circ\Psi^{(\alpha_m)})$.

Similary, to prove (\ref{duo}) observe that $\Psi\Sigma\subset(\Psi^{(\frac{1}{2})}\circ\Psi^{(\frac{1}{2})})(\Sigma^{(\frac{1}{2})}\circ\Sigma^
{(\frac{1}{2})})$, since $AB=(A^{(\frac{1}{2})}\circ A^{(\frac{1}{2})})(B^{(\frac{1}{2})}\circ B^{(\frac{1}{2})})$ for all $A\in\Psi$ and $B\in\Sigma$. It follows that 
\begin{eqnarray}
\nonumber
r(\Psi\Sigma)&\le &r((\Psi^{(\frac{1}{2})}\circ\Psi^{(\frac{1}{2})})(\Sigma^{(\frac{1}{2})}
\circ\Sigma^{(\frac{1}{2})}))\\
&\le & r \left((\Psi \Sigma)^{\left(\frac{1}{2}\right)} \circ (\Psi \Sigma)^{\left(\frac{1}{2}\right)}\right) ^{\beta}   r \left((\Sigma\Psi)^{\left(\frac{1}{2}\right)} \circ ( \Sigma \Psi)^{\left(\frac{1}{2}\right)}\right) ^{1-\beta} \le r(\Psi\Sigma)
\nonumber
\end{eqnarray}
by (\ref{AudSchPep}), which proves (\ref{duo}). It is proved similarly that (\ref{uno}) and (\ref{duo}) hold also for each $r\in \{ \rho _{ess}, \hat{\rho} _{ess}\}$ in the case when $L$ and $L^*$ have order continuous norms.

(ii) For the proof of (\ref{tre}) observe that $\Psi^{(m)}\subset\Psi\circ\cdots\circ\Psi$, since $A^{(m)}=A\circ\cdots\circ A$ for all $A\in\Psi$. By Theorem \ref{for_matrices}, Inequalities (\ref{tre}) follow. Inequalities (\ref{tre1}) are proved in a similar way. 
\end{proof}
\begin{remark}{\rm
Equalities (\ref{duo}) show that  the third inequality in (\ref{HZPep}) and the second and third inequality in (\ref{AudSchPep})  are in fact equalities.

This also implies that (only) \cite[Remark 3.6]{P19} is false. Indeed, 
\cite[Example 3.11]{P12} is not an example  that would support the claim stated in \cite[Remark 3.6]{P19}. The second author of this article regrets for  stating this false remark in \cite{P19}.
}
\end{remark}

The following result extends Inequalities (\ref{tri}) and (\ref{gsh_ref}) and Theorem \ref{for_matrices}.
\begin{theorem}
\label{finally}
Let $\{\Psi _{i j}\}_{i=1, j=1}^{k, m}$ be bounded sets of positive kernel operators on a Banach function space $L$ and let
 $\alpha _1, \ldots , \alpha _m$ be positive numbers.   \\

\noindent (i) If  $r \in \{\rho, \hat{\rho}\}$, $\sum _{i=1} ^m \alpha _i = 1$ and $n \in \NN$, then
\begin{eqnarray}
\nonumber
& & r \left(\left(\Psi_{1 1}^{(\alpha _1)} \circ \cdots \circ \Psi_{1 m}^{(\alpha _m)}\right) \ldots \left(\Psi_{k 1}^{(\alpha _1)} \circ \cdots \circ \Psi_{k m}^{(\alpha _m)} \right) \right) \\
\nonumber
& \le &r \left((\Psi_{1 1} \cdots  \Psi_{k 1})^{(\alpha _1)} \circ \cdots
\circ (\Psi_{1 m} \cdots \Psi_{k m})^{(\alpha _m)}\right)\\
\nonumber
& \le &r \left(((\Psi_{1 1} \cdots  \Psi_{k 1})^n)^{(\alpha _1)} \circ \cdots
\circ ((\Psi_{1 m} \cdots \Psi_{k m})^n)^{(\alpha _m)}\right)^{\frac{1}{n}}\\
&\le&r \left( \Psi_{1 1} \cdots  \Psi_{k 1} \right)^{\alpha _1} \cdots
r \left( \Psi_{1 m} \cdots \Psi_{k m}\right)^{\alpha _m} .
\label{lepa}
\end{eqnarray}
If, in addition, $L$ and $L^*$ have order continuous norms, then Inequalities (\ref{lepa}) hold also for each $r\in \{ \rho _{ess}, \hat{\rho} _{ess}\}$.

\noindent(ii) If $L\in\mathcal L$, $r \in \{\rho, \hat{\rho}\}$, $\sum_{j=1}^{m} \alpha _j\ge 1$ and $\{\Psi_{i j}\}_{i=1, j=1}^{k, m}$ are bounded sets of nonnegative matrices that define positive operators on $L$, then Inequalities (\ref{lepa}) hold. 

In particular, if $\Psi _1, \ldots, \Psi_k$ are bounded sets of nonnegative matrices that define positive operators on $L$ and
 $t\ge1$, then
\be
r (\Psi _1 ^{(t)} \cdots \Psi _k ^{(t)})\le r ((\Psi _1 \cdots \Psi _k) ^{(t)}) \le r(((\Psi _1 \cdots \Psi _k)^n)^{(t)})^{\frac{1}{n}}\le r(\Psi _1 \cdots \Psi _k)^{t}.
\label{with_t}
\ee 
\end{theorem}
\begin{proof} (i) Let  $r \in \{\rho, \hat{\rho}\}$, $\sum _{i=1} ^m \alpha _i = 1$ and $n \in \NN$. To prove the first inequality in (\ref{lepa}) let $l\in \NN$ and 
$$A \in \left(\left(\Psi_{1 1}^{(\alpha _1)} \circ \cdots \circ \Psi_{1 m}^{(\alpha _m)}\right) \ldots \left(\Psi_{k 1}^{(\alpha _1)} \circ \cdots \circ \Psi_{k m}^{(\alpha _m)} \right)\right)^l.$$
Then $A=A_1 \cdots A_l$, where for each $i=1, \ldots , l$ we have
$$A_i = \left(A_{i1 1}^{(\alpha _1)} \circ \cdots \circ A_{i 1 m}^{(\alpha _m)}\right) \ldots \left(A_{i k 1}^{(\alpha _1)} \circ \cdots \circ A_{i k m}^{(\alpha _m)} \right),$$
where $A_{i1 1} \in \Psi_{1 1}, \ldots , A_{i1 m} \in \Psi_{1 m}, \ldots , A_{ik 1} \in \Psi_{k 1}, \ldots , A_{ikm} \in \Psi_{km}. $ 
Then by (\ref{norm2}) for each $i=1, \ldots , l$ we have
$$A_i \le C_i := (A_{i11}A_{i21} \cdots A_{ik1} )^{(\alpha_1)} \circ \cdots \circ (A_{i1m}A_{i2m} \cdots A_{ikm} )^{(\alpha_m)}, $$
where $C_i \in (\Psi_{1 1} \cdots  \Psi_{k 1})^{(\alpha _1)} \circ \cdots
\circ (\Psi_{1 m} \cdots \Psi_{k m})^{(\alpha _m)}.$ Therefore 
$$A \le C:=C_1 \cdots C_l \in \left((\Psi_{1 1} \cdots  \Psi_{k 1})^{(\alpha _1)} \circ \cdots
\circ (\Psi_{1 m} \cdots \Psi_{k m})^{(\alpha _m)}\right)^l,$$
$\rho(A)^{1/l}\le \rho(C)^{1/l}$ and $\|A\|^{1/l}\le \|C\|^{1/l}$, which implies the first inequality in (\ref{lepa}). The second and third inequality in (\ref{lepa}) follow from  (\ref{gsh_ref}).

If, in addition, $L$ and $L^*$ have order continuous norms and $r\in \{ \rho _{ess}, \hat{\rho} _{ess}\}$, then Inequalities (\ref{lepa}) are proved similarly. Under the assumptions of (ii) Inequalities (\ref{lepa}) are proved in a similar way by applying Theorems \ref{thbegin}(ii) and \ref{for_matrices}.
\end{proof}

Next we extend  Theorem \ref{refin} by refining (\ref{Hu_ess}). 
\begin{theorem}
\label{kathyth1}
 Let $\Psi _1, \ldots , \Psi _m$ be bounded sets of positive kernel operators on a Banach function space $L$ and let
 $ \Phi _j=\Psi_j\ldots\Psi_m \Psi_1\ldots\Psi_{j-1}$ for $j=1,\ldots , m$. 
 If $r\in \{ \rho, \hat{\rho}\}$, then 
\begin{eqnarray}
\nonumber
& & r \left(\Psi _1 ^{\left( \frac{1}{m} \right)} \circ\Psi _2^{\left( \frac{1}{m} \right)}\circ\cdots\circ \Psi _m  ^{\left( \frac{1}{m} \right)} \right)
 \le r \left(\Phi _1^{\left( \frac{1}{m} \right)} \circ\Phi_2^{\left( \frac{1}{m} \right)}\circ\cdots\circ \Phi _m ^{\left( \frac{1}{m} \right)} \right)^{\frac{1}{m}}\\
&\le & r\left(\left(\Phi_1^{n}\right)^{(\frac{1}{m})}\circ (\Phi_2^{n})^{(\frac{1}{m})} \circ \cdots\circ
\left(\Phi_m^{n}\right)^{(\frac{1}{m})}\right)^{\frac{1}{nm}}\le r(\Psi_1 \Psi _2\cdots\Psi_m)^{\frac{1}{m}}.
\label{ineq1}
\end{eqnarray}

If, in addition, $L$ and $L^*$ have order continuous norms, then Inequalities (\ref{ineq1}) are valid also for all $r\in \{\rho_{ess}, \hat{\rho} _{ess}\}$.
\label{REFfr}
\end{theorem}
\begin{proof}
Let $r\in\{  \rho, \hat{\rho}\}$. Denote
\begin{eqnarray}
\nonumber
\Sigma _1 &= & \Psi _1 ^{\left( \frac{1}{m} \right)}  \circ\cdots\circ \Psi _m  ^{\left( \frac{1}{m} \right)} , \;\;\; \Sigma _2= \Psi _2 ^{\left( \frac{1}{m} \right)}  \circ\cdots\circ \Psi _m  ^{\left( \frac{1}{m}\right)} \circ \Psi _1 ^{\left( \frac{1}{m} \right)}, \ldots, \\
\nonumber
\Sigma _m &=&\Psi _m ^{\left( \frac{1}{m} \right)} \circ \Psi _1 ^{\left( \frac{1}{m} \right)} \circ\cdots\circ \Psi _{m-1}  ^{\left( \frac{1}{m} \right)}.
\end{eqnarray}
Then by  (\ref{again}), (\ref{lepa}) and commutativity of Hadamard product we have

\begin{eqnarray}
\nonumber
& & r \left(\Psi _1 ^{\left( \frac{1}{m} \right)} \circ\Psi _2^{\left( \frac{1}{m} \right)}\circ\cdots\circ \Psi _m  ^{\left( \frac{1}{m} \right)} \right)^m = 
 r \left(\left(\Psi _1 ^{\left( \frac{1}{m} \right)} \circ\Psi _2^{\left( \frac{1}{m} \right)}\circ\cdots\circ \Psi _m  ^{\left( \frac{1}{m} \right)} \right)^m\right)  \\
 \nonumber
 &=& r (\Sigma _1 \Sigma_2 \cdots \Sigma _m)\le r \left(\Phi _1^{\left( \frac{1}{m} \right)} \circ\Phi_2^{\left( \frac{1}{m} \right)}\circ\cdots\circ \Phi _m ^{\left( \frac{1}{m} \right)} \right),
\end{eqnarray}
which proves the first inequality in (\ref{ineq1}).
 The second and the third inequality in (\ref{ineq1}) follow from (\ref{gsh_ref}) (or from (\ref{lepa})), since $r(\Phi_1)=r(\Phi_2)=\cdots r(\Phi_m)=r(\Psi_1\Psi_2\cdots\Psi_m)$ by (\ref{again}). If, in addition, $L$ and $L^*$ have order continuous norms, then (\ref{ineq1}) for $r\in \{ \rho _{ess}, \hat{\rho} _{ess}\}$ is proved in a similar way.
 \end{proof}

The following result extends (\ref{duo}).

\begin{theorem}
\label{equalities_joint}
Let $\Psi _1, \ldots , \Psi _m$ be bounded sets of positive kernel operators on a Banach function space $L$ and let $\alpha_1, \ldots ,\alpha_m$ be nonnegative numbers such that $\sum_{j=1}^{m}\alpha_j=1$. If  $ \Phi _j=\Psi_j\ldots\Psi_m \Psi_1\ldots\Psi_{j-1}$ for $j=1,\ldots , m$, $\beta\in [0,1]$, then for all $r\in \{ \rho, \hat{\rho}\}$ we have
\begin{eqnarray}
\nonumber
 r\left(\Psi_1\Psi_2\cdots\Psi_m\right)&=&r\left(\left(\Psi_1^{(\beta)}\circ\Psi_1^{(1-\beta)}\right)
\cdots\left(\Psi_m^{(\beta)}\circ\Psi_m^{(1-\beta)}\right)\right)\\
&=&
r\left(\Phi_1^{(\beta)}\circ\Phi_1^{(1-\beta)}\right)^{\alpha_1}\cdots
r\left(\Phi_m^{(\beta)}\circ\Phi_m^{(1-\beta)}\right)^{\alpha_m}.
\label{prop2}
\end{eqnarray}
If, in addition, $L$ and $L^*$ have order continuous norms, then Equalities (\ref{prop2}) are valid for $r\in \{\rho_{ess}, \hat{\rho} _{ess}\}$.
\end{theorem}
\begin{proof}
Let $r\in \{ \rho, \hat{\rho}\}$. To prove Equalities (\ref{prop2}) we use the first inequality in (\ref{lepa}) and (\ref{again}) 
to obtain that 
\be
r\left(\left(\Psi_1^{(\beta)}\circ\Psi_1^{(1-\beta)}\right)
\cdots\left(\Psi_m^{(\beta)}\circ\Psi_m^{(1-\beta)}\right)\right)\le r\left(\Phi_i^{(\beta)}\circ\Phi_i^{(1-\beta)}\right)
\label{protesti}
\ee
 for all $i=1,\ldots , m$. Indeed, 
by (\ref{again}) and the first inequality in (\ref{lepa}) we have
$$r((\Psi_1^{(\beta)}\circ\Psi_1^{(1-\beta)})
\cdots(\Psi_m^{(\beta)}\circ\Psi_m^{(1-\beta)})) $$
$$= r((\Psi_i^{(\beta)}\circ\Psi_i^{(1-\beta)})
\cdots(\Psi_m^{(\beta)}\circ\Psi_m^{(1-\beta)})(\Psi_1^{(\beta)}\circ\Psi_1^{(1-\beta)})\cdots(\Psi_{i-1}^{(\beta)}\circ\Psi_{i-1}^{(1-\beta)}))$$
$$\le  r(\Phi_i^{(\beta)}\circ\Phi_i^{(1-\beta)}),$$
which proves (\ref{protesti}).
Since $\sum_{j=1}^{m}\alpha_j=1$, Inequality (\ref{protesti}) implies 
\be
\nonumber
r((\Psi_1^{(\beta)}\circ\Psi_1^{(1-\beta)})
\cdots(\Psi_m^{(\beta)}\circ\Psi_m^{(1-\beta)}))\le r(\Phi_1^{(\beta)}\circ\Phi_1^{(1-\beta)})^{\alpha_1}\cdots r(\Phi_m^{(\beta)}\circ\Phi_m^{(1-\beta)})^{\alpha_m}
\ee
\be
\le r(\Psi_1\cdots\Psi_m). 
\label{keineimpfpflicht}
\ee
The second inequality in (\ref{keineimpfpflicht}) follows from (\ref{gsh_ref}) and the fact that $r(\Phi_1)=\cdots=r(\Phi_m)=r(\Psi_1\cdots\Psi_m)$. Since $\Psi_i\subset \Psi_i^{(\beta)}\circ\Psi_i^{(1-\beta)}$ for all $i=1,\ldots , m$ and $\beta\in [0,1]$, we obtain 
$$ r(\Psi_1\cdots\Psi_m) \le r((\Psi_1^{(\beta)}\circ\Psi_1^{(1-\beta)})\cdots(\Psi_m^{(\beta)}\circ\Psi_m^{(1-\beta)})),$$
which together with  (\ref{keineimpfpflicht}) proves Equalities (\ref{prop2}).
If, in addition, $L$ and $L^*$ have order continuous norms, then Equalities (\ref{prop2}) are proved in a similar way for $r\in \{\rho_{ess}, \hat{\rho} _{ess}\}$.
\end{proof}


The following result, that extends the main results from \cite{P12}, is proved in  a similar way as Theorem \ref{REFfr} by applying Theorems \ref{for_matrices} and \ref{finally}(ii)  instead of Theorems \ref{powers} and \ref{finally}(i) in the proofs above.
\begin{theorem}
\label{kathyth2}
Given $L\in\mathcal{L}$, let $\Psi _1, \ldots , \Psi _m$ be bounded sets of nonnegative matrices that define operators on L and
 $ \Phi _j=\Psi_j\ldots\Psi_m \Psi_1\ldots\Psi_{j-1}$ for $j=1,\ldots , m$. Assume that
 $\alpha\ge\frac{1}{m}$, $\alpha _j \ge 0$, $j=1, \ldots , m$, $\sum_{j=1}^{m}\alpha_j \ge 1$ and $n \in \NN$. If $r\in \{ \rho, \hat{\rho}\}$ and $ \Sigma _j=\Psi_j ^{(\alpha m)}\ldots\Psi_m ^{(\alpha m)} \Psi_1 ^{(\alpha m)} \ldots\Psi_{j-1} ^{(\alpha m)}$ for $j=1,\ldots , m$, then we have
\begin{eqnarray}
\nonumber
& & r\left(\Psi _1^{(\alpha)}\circ\cdots\circ\Psi _m^{(\alpha)}\right)\le r\left(\Phi _1^{(\alpha)}\circ\cdots\circ\Phi _m^{(\alpha)}\right)^{\frac{1}{m}} \\
&\le& r\left((\Phi _1 ^n)^{(\alpha)}\circ\cdots\circ(\Phi _m ^n)^{(\alpha)}\right)^{\frac{1}{mn}} \le r\left(\Psi_1\cdots\Psi_m\right)^{\alpha},
\label{ineqx}
\end{eqnarray}
\begin{eqnarray}
\nonumber
& & r\left(\Psi _1^{(\alpha)}\circ\cdots\circ\Psi _m^{(\alpha)}\right)\le r\left(\Psi_1^{(\alpha m)}\cdots\Psi_m^{(\alpha m)}\right)^{\frac{1}{m}}\\
&\le&  r\left((\Psi_1\cdots\Psi_m)^{(\alpha m)}\right)^{\frac{1}{m}}\le  r\left(((\Psi_1\cdots\Psi_m)^n)^{(\alpha m)}\right)^{\frac{1}{nm}} \le
r\left(\Psi_1\cdots\Psi_m\right)^{\alpha}. \;\;\;\;\;\;
\label{ineqx2}
\end{eqnarray}
If, in addition, $\alpha \ge 1$ then 
\begin{eqnarray}
\nonumber
& & r\left(\Psi_1^{(\alpha)}\circ\cdots\circ\Psi_m^{(\alpha)}\right)  \le  r\left(\Phi_1^{(\alpha)}\circ\cdots\circ
\Phi_m^{(\alpha)}\right)^{\frac{1}{m}} \le  r\left((\Phi _1 ^n)^{(\alpha)}\circ\cdots\circ(\Phi _m ^n)^{(\alpha)}\right)^{\frac{1}{mn}} \\
&\le &\left( r\left((\Phi_1 ^n)^{(m)}\right)\cdots r\left((\Phi_m ^n)^{(m)}\right)\right)^{\frac{\alpha}{m^2 n}}
 \le r\left(\Psi_1\cdots\Psi_m\right)^{\alpha},
\label{xyz}
\end{eqnarray}
\begin{eqnarray}
\nonumber
& & r\left(\Psi _1^{(\alpha)}\circ\cdots\circ\Psi _m^{(\alpha)}\right)\le  r\left(\Sigma _1^{(\frac{1}{m})}\circ\cdots\circ\Sigma _m^{(\frac{1}{m})}\right)^{\frac{1}{m}} \\
\nonumber
&\le &  r\left((\Sigma _1 ^n)^{(\frac{1}{m})}\circ\cdots\circ(\Sigma _m ^n)^{(\frac{1}{m})}\right)^{\frac{1}{mn}}
 \le r\left(\Psi_1^{(\alpha m)}\cdots\Psi_m^{(\alpha m)}\right)^{\frac{1}{m}} \\
&\le &  r\left((\Psi_1\cdots\Psi_m)^{(\alpha m)}\right)^{\frac{1}{m}}\le  r\left(((\Psi_1\cdots\Psi_m)^n)^{(\alpha m)}\right)^{\frac{1}{nm}} \le r\left(\Psi_1\cdots\Psi_m\right)^{\alpha}. \;\;\;\;\;\;
\label{kraj}
\end{eqnarray}
\end{theorem}
\begin{proof}
Inequalities (\ref{ineqx}) are proved in a similar way as Theorem \ref{kathyth1} by applying  Theorems \ref{for_matrices} and \ref{finally}(ii)  instead of Theorems \ref{powers} and \ref{finally}(i). For the proof of (\ref{ineqx2}) observe that 
$$\Psi_1^{(\alpha)}\circ\cdots\circ\Psi_m^{(\alpha)}=(\Psi_1^{(\alpha m)})^{(\frac{1}{m})}\circ\cdots\circ(\Psi_m^{(\alpha m)})^{(\frac{1}{m})}$$
 for $i=1,\ldots , m$. Now the first inequality in (\ref{ineqx2}) follows from  (\ref{Hu_ess})
 (or from (\ref{xyz})): 
\be
\nonumber
r\left(\Psi_1^{(\alpha)}\circ\cdots\circ\Psi_m^{(\alpha)}\right)=r\left((\Psi_1^{(\alpha m)})^{(\frac{1}{m})}\circ\cdots\circ(\Psi_m^{(\alpha m)})^{(\frac{1}{m})}\right)\le r\left(\Psi_1^{(\alpha m)}\cdots\Psi_m^{(\alpha m)}\right)^{\frac{1}{m}}.
\ee
Other inequalities in  (\ref{ineqx2}) follow from Theorem \ref{finally}(ii).

Assume $\alpha \ge 1$. The first and second inequality in (\ref{xyz}) follow from (\ref{ineqx}).  To prove the third inequality in (\ref{xyz}) notice that $(\Phi_i ^n)^{(\alpha)}=((\Phi_i ^n)^{(m)})^{(\frac{\alpha}{m})}$, $\frac{\alpha}{m} \ge \frac{1}{m}$ and apply 
Theorem \ref{for_matrices}. The last inequality in (\ref{xyz}) follows again from Theorem \ref{for_matrices} and the fact that $r(\Phi_1)= \cdots = r(\Phi_m)=r(\Psi_1\cdots\Psi_m)$.

To prove the first three inequalities in (\ref{kraj}) observe that  $\Psi_i^{(\alpha)}=(\Psi_i^{(m \alpha)})^{(\frac{1}{m})}$, $\frac{\alpha}{m} \ge \frac{1}{m}$ and apply Theorem \ref{kathyth1}. The remaining three inequalities in  (\ref{kraj}) follow from (\ref{ineqx2}), which completes the proof.
\end{proof}

We will need the following well-known inequalities (see e.g. \cite{Mi}). For
nonnegative measurable functions and for nonnegative numbers $\alpha$ and $\beta$  such that $\alpha+\beta \ge 1$ we have
\be
f_1^{\alpha}g_1^{\beta}+\cdots+f_m^{\alpha}g_m^{\beta} \le (f_1+\cdots+f_m)^{\alpha}(g_1+\cdots+g_m)^{\beta}
\label{mitrn} 
\ee
More generally, for nonnegative measurable functions $\{f _{i j}\}_{i=1, j=1}^{k, m}$ and for nonnegative numbers  $\alpha_j$, $j=1, \ldots , m$,   such that $\sum _{j=1} ^m \alpha _j \ge 1$ we have
\be
(f_{11}^{\alpha _1} \cdots f_{1m}^{\alpha _m})+\cdots+(f_{k1}^{\alpha _1} \cdots f_{km}^{\alpha _m}) \le (f_{11}+\cdots+f_{k1})^{\alpha _1} \cdots (f_{1m}+\cdots+f_{km})^{\alpha _m}
\label{mitr2}
\ee
The sum of bounded sets $\Psi$ and $\Sigma$ is a bounded set defined by
$\Psi + \Sigma =\{A+B: A \in \Psi, B \in \Sigma\}$.

\begin{theorem}
\label{finally2}
Let $\{\Psi _{i j}\}_{i=1, j=1}^{k, m}$ be bounded sets of positive kernel operators on a Banach function space $L$ and let
 $\alpha _1, \ldots , \alpha _m$ be positive numbers.   \\

\noindent (i) If  $r \in \{\rho, \hat{\rho}\}$, $\sum _{j=1} ^m \alpha _j = 1$ and $n \in \NN$, then
\begin{eqnarray}
\nonumber
& & r \left(\left(\Psi_{1 1}^{(\alpha _1)} \circ \cdots \circ \Psi_{1 m}^{(\alpha _m)}\right) + \ldots + \left(\Psi_{k 1}^{(\alpha _1)} \circ \cdots \circ \Psi_{k m}^{(\alpha _m)} \right) \right) \\
\nonumber
& \le &r \left((\Psi_{1 1}+ \cdots  + \Psi_{k 1})^{(\alpha _1)} \circ \cdots
\circ (\Psi_{1 m} + \cdots  + \Psi_{k m})^{(\alpha _m)}\right)\\
\nonumber
& \le &r \left(((\Psi_{1 1} + \cdots  + \Psi_{k 1})^n)^{(\alpha _1)} \circ \cdots
\circ ((\Psi_{1 m} + \cdots + \Psi_{k m})^n)^{(\alpha _m)}\right)^{\frac{1}{n}}\\
&\le&r \left( \Psi_{1 1} + \cdots  + \Psi_{k 1} \right)^{\alpha _1} \cdots
r \left( \Psi_{1 m} + \cdots  + \Psi_{k m}\right)^{\alpha _m} .
\label{lepa2}
\end{eqnarray}
If, in addition, $L$ and $L^*$ have order continuous norms, then Inequalities (\ref{lepa2}) hold also for each $r\in \{ \rho _{ess}, \hat{\rho} _{ess}\}$.

\noindent(ii) If $L\in\mathcal L$, $r \in \{\rho, \hat{\rho}\}$, $\sum_{j=1}^{m} \alpha _j\ge 1$ and $\{\Psi_{i j}\}_{i=1, j=1}^{k, m}$ are bounded sets of nonnegative matrices that define positive operators on $L$, then Inequalities (\ref{lepa2}) hold. 
\end{theorem}
\begin{proof} (i) Let  $r \in \{\rho, \hat{\rho}\}$, $\sum _{i=1} ^m \alpha _i = 1$ and $n \in \NN$. To prove the first inequality in (\ref{lepa2}) let $l\in \NN$ and 
$$A \in \left(\left(\Psi_{1 1}^{(\alpha _1)} \circ \cdots \circ \Psi_{1 m}^{(\alpha _m)}\right) +\ldots +\left(\Psi_{k 1}^{(\alpha _1)} \circ \cdots \circ \Psi_{k m}^{(\alpha _m)} \right)\right)^l.$$
Then $A=A_1 \cdots A_l$, where for each $i=1, \ldots , l$ we have
$$A_i = \left(A_{i1 1}^{(\alpha _1)} \circ \cdots \circ A_{i 1 m}^{(\alpha _m)}\right) +\ldots + \left(A_{i k 1}^{(\alpha _1)} \circ \cdots \circ A_{i k m}^{(\alpha _m)} \right),$$
where $A_{i1 1} \in \Psi_{1 1}, \ldots , A_{i1 m} \in \Psi_{1 m}, \ldots , A_{ik 1} \in \Psi_{k 1}, \ldots , A_{ikm} \in \Psi_{km}. $ 
Then by (\ref{mitr2}) for each $i=1, \ldots , l$ we have
$$A_i \le C_i := (A_{i11}+A_{i21}+ \cdots + A_{ik1} )^{(\alpha_1)} \circ \cdots \circ (A_{i1m}+A_{i2m}+ \cdots +A_{ikm} )^{(\alpha_m)}, $$
where $C_i \in (\Psi_{1 1} +\cdots  +\Psi_{k 1})^{(\alpha _1)} \circ \cdots
\circ (\Psi_{1 m} +\cdots +\Psi_{k m})^{(\alpha _m)}.$ Therefore 
$$A \le C:=C_1\cdots C_l  \in \left((\Psi_{1 1} +\cdots  +\Psi_{k 1})^{(\alpha _1)} \circ \cdots
\circ (\Psi_{1 m} + \cdots + \Psi_{k m})^{(\alpha _m)}\right)^l,$$
$r(A)^{1/l}\le r(C)^{1/l}$ and $\|A\|^{1/l}\le \|C\|^{1/l}$, which implies the first inequality in (\ref{lepa2}). The second and third inequality in (\ref{lepa2}) follow from  (\ref{gsh_ref}).

If, in addition, $L$ and $L^*$ have order continuous norms and $r\in \{ \rho _{ess}, \hat{\rho} _{ess}\}$, then Inequalities (\ref{lepa2}) are proved similarly. Under the assumptions of (ii) Inequalities (\ref{lepa2}) are proved in a similar way by applying Theorems \ref{thbegin}(ii) and \ref{for_matrices}.
\end{proof}

\section{Weighted geometric symmetrizations}

Let $\Psi$ and $\Sigma$ be bounded sets of positive kernel operators on $L^2(X,\mu)$ and $\alpha \in [0,1]$. Denote by  $\Psi ^*$ and $S_{\alpha}(\Psi)$  bounded sets of positive kernel operators on $L^2(X,\mu)$ defined by 
$\Psi ^* =\{A^* : A \in \Psi\}$ and
$$S_{\alpha}(\Psi) = \Psi ^{(\alpha)} \circ (\Psi ^*)^{(1-\alpha)} =\{ A ^{(\alpha)} \circ (B^*)^{(1-\alpha)}:A,B \in \Psi\}.$$
We denote simply $S(\Psi)=S_{\frac{1}{2}}(\Psi)$, {\it the geometric symmetrization} of $\Psi$.
Observe that $(\Psi \Sigma)^*= \Sigma ^* \Psi ^*$ and $(\Psi ^m)^* = (\Psi ^*)^m$ for all $m \in \NN$.
By (\ref{gsh_ref}) it follows that
\be
r(S_{\alpha}(\Psi)) \le  r (S_{\alpha}( \Psi ^{m}))^{\frac{1}{m}} \le r(\Psi)
\label{good_corollary_2}
\ee
for all $m \in \mathbb{N}$ and $r\in \{ \rho , \hat{\rho} , \rho _{ess}, \hat{\rho} _{ess}\}$, since $r(\Psi) = r(\Psi ^*)$.
In particular, for all $r\in \{ \rho , \hat{\rho} , \rho _{ess}, \hat{\rho} _{ess}\}$ and 
$n \in \mathbb{N}\cup \{0\}$ we have 
\be
r(S_{\alpha}(\Psi)) \le  r (S_{\alpha}( \Psi ^{2^n}))^{2^{-n}} \le r(\Psi).
\label{good_corollary}
\ee
 Consequently,
\be
r (S_{\alpha}(\Psi ))^2 \le r (S_{\alpha}(\Psi ^2)) \le r(\Psi)^2
\label{good_corollary2}
\ee
holds for all $r\in \{ \rho , \hat{\rho} , \rho _{ess}, \hat{\rho} _{ess}\}$.

The following result that follows from (\ref{good_corollary}) extends 
\cite[Theorem 2.2]{Drn}, \cite[Theorem 3.5]{P21} and \cite[Theorem 3.5(i)]{BP21}. 
\begin{theorem}
\label{...+drnovsek}
 Let $\Psi$ be a bounded set of positive kernel operators on $L^2(X,\mu)$, $\alpha \in [0,1]$ and let 
$r_n = r (S_{\alpha}( \Psi ^{2^n}))^{2^{-n}} $ for $n \in \mathbb{N}\cup \{0\}$ and 
$r\in \{ \rho , \hat{\rho} , \rho _{ess}, \hat{\rho} _{ess}\}$. 
Then for each $n$
$$ r (S_{\alpha}(\Psi))= r_0 \le r_1 \le \cdots \le r_n \le r (\Psi).$$	
\end{theorem}
\begin{proof}
By (\ref{good_corollary}) we have 
$r_n \le r(\Psi)$. Since $r_{n-1}\le r_n$ for all $n \in \mathbb{N}$ by the first inequality in  (\ref{good_corollary2}), the proof is completed.
\end{proof}

The following result extends  \cite[Proposition 3.2]{BP21}.
\begin{proposition}
Let $\Psi_1, \ldots ,\Psi_m$ be bounded sets of positive kernel operators on $L^2(X,\mu)$, $\alpha \in [0,1]$, $n \in \NN$ and $r\in \{ \rho , \hat{\rho} , \rho _{ess}, \hat{\rho} _{ess}\}$. Then we have
\begin{eqnarray}
\nonumber
& & r(S_{\alpha} (\Psi_{1}) \cdots  S_{\alpha}(\Psi_{m})) 
\le r\left((\Psi_1 \cdots \Psi_m )^{(\alpha)} \circ ((\Psi_m \cdots \Psi_1)^*)^{(1-\alpha)}\right) \\
\nonumber
&\le & r\left(((\Psi_1 \cdots \Psi_m )^n)^{(\alpha)} \circ (((\Psi_m \cdots \Psi_1)^*)^n)^{(1-\alpha)}\right)^{\frac{1}{n}}\\
&\le & r(\Psi_1 \cdots \Psi_m )^{\alpha} \, r(\Psi_m \cdots \Psi_1 )^{1-\alpha} ,
\label{freiheit}
\end{eqnarray}
\begin{eqnarray}
\nonumber
& & r(S_{\alpha}(\Psi_1)+ \cdots + S_{\alpha}(\Psi_m)) \le r\left(S_{\alpha}(\Psi_1+ \cdots+ \Psi_m)\right) \\
&\le & r\left(S_{\alpha}((\Psi_1+ \cdots+ \Psi_m)^{n})\right)^{\frac{1}{n}}
\le r(\Psi_1+ \cdots+ \Psi_m).
\label{widerstand}
\end{eqnarray}

In particular, we have
\begin{eqnarray}
\nonumber
& & r\left(S_{\alpha}(\Psi_{1}) S_{\alpha}(\Psi_{2})\right) \le r\left((\Psi_1 \Psi_2 )^{(\alpha)} \circ ((\Psi_2 \Psi_1)^*)^{(1-\alpha)}  \right) \\
&\le & r\left(((\Psi_1 \Psi_2 )^n)^{(\alpha)} \circ (((\Psi_2 \Psi_1)^*)^n)^{(1-\alpha)}\right)^{\frac{1}{n}}\le r(\Psi_1 \Psi_2) .
\label{special}
\end{eqnarray}
\label {geom_sym}
\end{proposition}
\begin{proof}
By Theorem \ref{finally}(i) we have
\begin{eqnarray}
\nonumber
& & r\left(S_{\alpha} (\Psi_{1}) \cdots  S_{\alpha}(\Psi_{m})\right) =r\left( (\Psi_1 ^{(\alpha)} \circ (\Psi_1^*)^{(1-\alpha)}) \cdots
\left(\Psi_m ^{(\alpha)} \circ (\Psi_m^*)^{(1-\alpha)}\right)\right) \\
\nonumber
&\le &  r\left((\Psi_1 \cdots \Psi_m )^{(\alpha)} \circ ((\Psi_m \cdots \Psi_1)^*)^{(1-\alpha)}  \right)\\
\nonumber
&\le & r\left(((\Psi_1 \cdots \Psi_m )^n)^{(\alpha)} \circ (((\Psi_m \cdots \Psi_1)^*)^n)^{(1-\alpha)}\right)^{\frac{1}{n}}\\
\nonumber
& \le&  r(\Psi_1 \cdots \Psi_m)^{\alpha} \,  r((\Psi_m \cdots \Psi_1)^*)^{1-\alpha} =
 r(\Psi_1 \cdots \Psi_m )^{\alpha} \,  r(\Psi_m \cdots \Psi_1 )^{1-\alpha},
 \end{eqnarray}
  where the last equality follows from  the fact that $r(\Psi)=r(\Psi^*)$. 
  The inequalities in (\ref{widerstand}) are proved in similar way by applying Theorem \ref{finally2} 
  and (\ref{good_corollary}). The first and second inequalities in (\ref{special}) are special cases of (\ref{freiheit}), while the third inequality follows from (\ref{freiheit}) and the fact that $r(\Psi_1\Psi_2)=r(\Psi_2\Psi_1)$.
\end{proof}
Let $\Psi$ be a bounded set of nonnegative matrices that define operators on $l^2(R)$ and let $\alpha$ and $\beta$ be nonnegative numbers such that $\alpha+\beta \ge 1$. The set $S_{\alpha, \beta}(\Psi)=\Psi^{(\alpha)}\circ(\Psi^*)^{(\beta)}=\{A^{(\alpha)}\circ (B^*)^{(\beta)} : A, B \in \Psi\}$ is a bounded set of nonnegative matrices that define operators on $l^2(R)$ by Theorem \ref{thbegin}(ii). 

For $r\in \{ \rho , \hat{\rho}\} $ the following result extends Theorem \ref{...+drnovsek} in the case of bounded set of nonnegative matrices that define operators on $l^2(R)$. It also extends a part of \cite[Theorem 3.5(ii)]{BP21}.
\begin{theorem}
\label{sym_matrices}
Let $\Psi$ be a bounded set of nonnegative matrices that define operators on $l^2(R)$ and $r \in \{ \rho, \hat{\rho}\}$. Assume $\alpha$ and $\beta$  are nonnegative numbers such that $\alpha+\beta \ge 1$ and denote $r_n=r(S_{\alpha,\beta}(\Psi^{2^n}))^{2^{-n}}$ for $n \in \mathbb{N}\cup \{0\}$. Then we have
\be
r(S_{\alpha, \beta}(\Psi))=r_{0}\le r_1\le \cdots \le r_n\le r(\Psi)^{\alpha+\beta}.
\label{finish}
\ee
\begin{proof}
 By Theorem \ref{for_matrices} we have
\be
r(S_{\alpha, \beta}(\Psi))=r(\Psi^{(\alpha)}\circ (\Psi^*)^{(\beta)})\le r\left((\Psi^{2^n})^{(\alpha)}\circ ((\Psi^*)^{2^n})^{(\beta)}\right)^{2^{-n}}=r_n\le r(\Psi)^{\alpha+\beta}.
\ee
In particular, for $n=1$ we have
\be
r(S_{\alpha, \beta}(\Psi))^2 \le r(S_{\alpha, \beta}(\Psi^2))\le r(\Psi)^{2(\alpha +\beta)}.
\label{tudi}
\ee
Since $r_{n-1} \le r_{n}$ for all $n \in \mathbb{N}\cup \{0\}$  by the first inequality in (\ref{tudi}), the proof of (\ref{finish}) is completed.
\end{proof}
\end{theorem}

The following result is proved in similar way as Proposition \ref{geom_sym} using Theorem \ref{finally}(ii) instead of Theorem \ref{finally}(i).
\begin{proposition}
Let $\Psi$, $\Psi_1, \ldots ,\Psi_m$ be bounded sets of nonnegative matrices that define operators on $l^2(R)$, $n \in \mathbb{N}$ and let $\alpha$ and $\beta$ be nonnegative numbers such that $\alpha+\beta \ge 1$. Then we have
\begin{eqnarray}
\nonumber
& & r(S_{\alpha ,\beta} (\Psi_{1}) \cdots  S_{\alpha  ,\beta}(\Psi_{m})) 
\le  r\left((\Psi_1 \cdots \Psi_m )^{(\alpha)} \circ ((\Psi_m \cdots \Psi_1)^*)^{(\beta)}  \right) \\
\nonumber
&\le & r\left(((\Psi_1 \cdots \Psi_m )^n)^{(\alpha)} \circ (((\Psi_m \cdots \Psi_1)^*)^n)^{(\beta)}\right)^{\frac{1}{n}}\\
&\le &  r(\Psi_1 \cdots \Psi_m )^{\alpha} \,  r(\Psi_m \cdots \Psi_1 )^{\beta},
\label{geom_sym_prva}
\end{eqnarray}
\be r(S_{\alpha, \beta}(\Psi)) \le r(S_{\alpha, \beta}(\Psi ^n))^{\frac{1}{n}}\le r(\Psi)^{\alpha+\beta},
\label{geom_sym_druga}
\ee
\begin{eqnarray}
\nonumber
& & r(S_{\alpha  ,\beta}(\Psi_1)+ \cdots + S_{\alpha  ,\beta}(\Psi_m)) \le r\left(S_{\alpha  ,\beta}(\Psi_1+ \cdots+ \Psi_m)\right) \\
&\le & r\left(S_{\alpha  ,\beta}((\Psi_1+ \cdots+ \Psi_m)^{n})\right)^{\frac{1}{n}} \le r(\Psi_1+ \cdots+ \Psi_m)^{\alpha + \beta},
\label{geom_sym_treca}
\end{eqnarray}
\begin{eqnarray}
\nonumber
& & r(S_{\alpha ,\beta}(\Psi_1)S_{\alpha ,\beta}(\Psi_2))\le r\left((\Psi_1 \Psi_2 )^{(\alpha)} \circ ((\Psi_2 \Psi_1)^*)^{(\beta)}  \right) \\
&\le & r\left(((\Psi_1 \Psi_2 )^n)^{(\alpha)} \circ (((\Psi_2 \Psi_1)^*)^n)^{(\beta)}\right)^{\frac{1}{n}}
\le r(\Psi_1\Psi_2)^{\alpha+\beta}
\label{geom_sym_cetvrta}
\end{eqnarray}
for $r \in \{ \rho, \hat{\rho}\}$.
\end{proposition}
\begin{proof}
Inequalities (\ref{geom_sym_prva}) and (\ref{geom_sym_treca}) are proved in a similar way as inequalities (\ref{freiheit}) and (\ref{widerstand}) by using Theorems \ref{finally}(ii) and \ref{finally2}(ii). 
Inequalities (\ref{geom_sym_druga}) and (\ref{geom_sym_cetvrta}) are special cases of (\ref{geom_sym_prva}).
\end{proof}

\bigskip

\noindent {\bf Acknowledgements.}
The first author acknowledges a partial support of Erasmus+ European Mobility program (grant KA103), COST Short Term Scientific Mission program (action CA18232) and the Slovenian Research Agency (grants P1-0222 and P1-0288). The first author thanks the colleagues and staff at the Faculty of Mechanical Engineering and Institute of Mathematics, Physics and Mechanics for their hospitality during the research stay in Slovenia. The second author acknowledges a partial support of  the Slovenian Research Agency (grants P1-0222, J1-8133 and J2-2512).

\bibliographystyle{amsplain}

\end{document}